\documentclass{amsart}
\usepackage{amsmath}
\usepackage{amssymb}
\usepackage{array}
\usepackage{amscd}
\usepackage{tikz-cd}
\usepackage{float}

\numberwithin{equation}{section}
\newtheorem{Th}[subsection]{Theorem}
\newtheorem*{Th*}{Theorem}
\newtheorem{Lemma}[subsection]{Lemma}
\newtheorem{Prop}[subsection]{Proposition}

\newtheorem{Cor}[subsection]{Corollary}
\theoremstyle{definition}
\newtheorem{definition}[subsection]{Definition}
\newtheorem*{definition*}{Definition}
\newtheorem{Remark}[subsection]{Remark}

\newcommand{\comm}[1]{}
\newcommand\sbullet[1][.5]{\mathbin{\vcenter{\hbox{\scalebox{#1}{$\bullet$}}}}}

\usepackage{color}
\usepackage[normalem]{ulem}
\definecolor{DarkGreen}{rgb}{0,0.5,0.1} 

\newcommand\soutD{\bgroup\markoverwith
{\textcolor{DarkGreen}{\rule[.5ex]{2pt}{1pt}}}\ULon}

\makeatletter
\newcommand{\romsmall}[1]{\romannumeral #1}
\newcommand*{\rom}[1]{\expandafter\@slowromancap\romannumeral #1@}
\makeatother

\begin{document}
 
\title[The most symmetric cubic surface over a  field of characteristic $2$]{The most symmetric smooth cubic surface over a~finite field of characteristic $2$}
\author{Anastasia V.~Vikulova}
\address{{\sloppy
\parbox{0.9\textwidth}{
Steklov Mathematical Institute of Russian
Academy of Sciences,
8 Gubkin str., Moscow, 119991, Russia
}\bigskip}}
\email{vikulovaav@gmail.com}
\date{}
\maketitle

\begin{abstract}
In this paper we find the largest automorphism group of a smooth cubic surface over any finite field of characteristic $2.$ We prove that if the order of the field is a power of~$4,$ then the automorphism group of maximal order of~a~smooth cubic surface over this field is $\mathrm{PSU}_4(\mathbb{F}_2).$ If  the order of the field of~characteristic~$2$ is not a power of $4,$ then we prove that the automorphism group of maximal order of a smooth cubic surface over this field is the symmetric group of degree $6.$ Moreover, we prove that smooth cubic surfaces with such properties are unique up to isomorphism.
\end{abstract}

\vspace{11pt}
	
	
	\vspace{11pt}


\section{Introduction}

The Fermat cubic surface 
$$
 x^3+y^3+z^3+t^3=0
$$
\noindent was interesting to many mathematicians since ancient times, at least since Porisms of Diophantus, where he mentioned  that the difference of two cubes can be represented as a sum of two cubes (see~\cite[Chapter \rom{21}]{DicksonHistory} for more historical details). It was studied a lot in number theory and algebraic geometry. The Fermat cubic surface has the largest automorphism group among smooth cubic surfaces over an algebraically closed field of characteristic not equal to $3.$ This was proved by T.~Hosoh  in \cite[Theorem 5.3]{Hosoh} for fields of characteristic zero, and  by  I.~Dolgachev and A.~Duncan in~\cite[Theorem 1.1]{DolgachevDuncan} for fields of positive characteristics. In particular, we have the following theorem:

\begin{Th}[{\cite[Theorem 1.1]{DolgachevDuncan}}]\label{Th:DolgachevDuncan}
Let $S$ be a smooth cubic surface over an algebraically closed field $\mathbf{F}$ of characteristic $2.$ Then $|\mathrm{Aut}(S)| \leqslant 25\, 920.$ The equality holds if and only if~$\mathrm{Aut}(S)\simeq \mathrm{PSU}_4(\mathbb{F}_2)$ and $S$ is isomorphic to the Fermat cubic surface.
\end{Th}

Note that from the proof of Theorem~\ref{Th:DolgachevDuncan} (see~\cite[\S5.1]{DolgachevDuncan} for more details) we get the following theorem about automorphism group of the Fermat cubic surface over~$\mathbb{F}_{4}.$

\begin{Th}[{\cite[\S5.1]{DolgachevDuncan}}]\label{th:autfermatf4}
Let $S$ be the Fermat cubic surface over $\mathbb{F}_4.$ Then we have 
$$
\mathrm{Aut}(S) \simeq \mathrm{PSU}_4(\mathbb{F}_2).
$$
\end{Th}

The automorphism group of a smooth cubic surface over a finite field can be much smaller than the automorphism group of a smooth cubic surface over an~algebraic closure of the field. For example, there is the following result.

\begin{Th}[{\cite[Theorem 3 and Example 9]{Vikulova}}]\label{th:cubicoverf2}
Let $S$ be a smooth cubic surface over the field $\mathbb{F}_2$ of $2$ elements. Then $|\mathrm{Aut}(S)| \leqslant 720.$ Moreover,  $|\mathrm{Aut}(S)| = 720$ if and only if $\mathrm{Aut}(S) \simeq \mathfrak{S}_6$ and  $S$ is  isomorphic to the cubic surface defined by the equation
\begin{equation}\label{eq:cubicsurfacef2}
x^2t+y^2z+z^2y+t^2x=0.
\end{equation}
\end{Th}

In this paper we find automorphism groups of smooth cubic surfaces  of maximal order over all finite fields $\mathbb{F}_{2^k}$  of order $2^k$ for all natural $k.$ So we prove the following theorem.

\begin{Th}\label{th:maximalaut2}
Let $S$ be a smooth cubic surface over a finite field $\mathbf{F}$ of characteristic~$2.$ 
\begin{enumerate}
\renewcommand\labelenumi{\rm (\roman{enumi})}
\renewcommand\theenumi{\rm (\roman{enumi})}

\item\label{th:maximalaut2F4^k}
If $\mathbf{F}=\mathbb{F}_{4^k},$ then $|\mathrm{Aut}(S)| \leqslant 25\, 920.$ Moreover, $|\mathrm{Aut}(S)| = 25\, 920$ if and only if
$$
\mathrm{Aut}(S) \simeq \mathrm{PSU}_4(\mathbb{F}_2)
$$
\noindent and $S$ is isomorphic to the Fermat cubic surface.

\item\label{th:maximalaut2F24^k}
If $\mathbf{F}=\mathbb{F}_{2^{2k+1}},$ then $|\mathrm{Aut}(S)| \leqslant 720.$ Moreover, $|\mathrm{Aut}(S)| = 720$ if and only if
$$
\mathrm{Aut}(S) \simeq \mathfrak{S}_6
$$
\noindent and $S$ is isomorphic to cubic surface~\eqref{eq:cubicsurfacef2}.

\end{enumerate}

\end{Th}

In the process of proving Theorem~\ref{th:maximalaut2} we get the following result.

\begin{Th}\label{th:maximalaut2z/5z}
Let $S$ be a smooth cubic surface over a finite field $\mathbf{F}$ of characteristic~$2.$ Denote by $\Gamma$ the image of the Galois group $\mathrm{Gal}(\overline{\mathbf{F}}/\mathbf{F})$ in the Weyl group~\mbox{$W(\mathrm{E}_6) \simeq \mathrm{Aut}(\mathrm{Pic}(S_{\overline{\mathbf{F}}})).$} Assume that 
$$
\mathbb{Z}/5\mathbb{Z} \subset \mathrm{Aut}(S).
$$
\begin{enumerate}
\renewcommand\labelenumi{\rm (\roman{enumi})}
\renewcommand\theenumi{\rm (\roman{enumi})}

\item\label{th:maximalaut2F4^kz/5z}
If $\mathbf{F}=\mathbb{F}_{4^k}$ and $\Gamma=\{e\},$ then  $S$ is isomorphic to the Fermat cubic surface.

\item\label{th:maximalaut2F24^kz/5z}
If $\mathbf{F}=\mathbb{F}_{2^{2k+1}}$ and $\Gamma \simeq \mathbb{Z}/2\mathbb{Z},$ then  $S$ is isomorphic to cubic surface~\eqref{eq:cubicsurfacef2}.

\end{enumerate}

\end{Th}

The following proposition shows that over fields $\mathbb{F}_{2^{2k+1}}$ the Fermat cubic surface and cubic~\eqref{eq:cubicsurfacef2} are not isomorphic whereas over $\overline{\mathbb{F}}_2$ they are isomorphic.

\begin{Prop}[{cf.~Theorem~\ref{th:maximalaut2}\ref{th:maximalaut2F4^k}}]\label{prop:Fermat}
Let $\mathcal{F}$ be the Fermat cubic surface over a~field~$\mathbf{F}$ of characteristic~$2.$ Then if $\mathbf{F}=\mathbb{F}_{2^{2k+1}},$ we get $\mathrm{Aut}(\mathcal{F}) \simeq \mathbb{Z}/2\mathbb{Z} \times \mathfrak{S}_4.$ 
\end{Prop}

\begin{Cor}\label{cor:Fermat}
Let $\mathbf{F}$ be  a field of characteristic $2.$ Then

\begin{enumerate}
\renewcommand\labelenumi{\rm (\roman{enumi})}
\renewcommand\theenumi{\rm (\roman{enumi})}

\item\label{prop:Fermat24^k}
if $\mathbf{F}=\mathbb{F}_{2^{2k+1}},$ then smooth cubic surface \eqref{eq:cubicsurfacef2} is not isomorphic to the Fermat cubic surface;

\item\label{prop:Fermat4^k}
if $\mathbf{F}=\mathbb{F}_{4^k}$ or $\mathbf{F}=\overline{\mathbf{F}},$ then smooth cubic surface \eqref{eq:cubicsurfacef2} and the Fermat cubic surface are isomorphic.

\end{enumerate}

\end{Cor}

Let us briefly describe the idea of the proof of Theorem~\ref{th:maximalaut2}.  We show that we can $\mathbb{Z}/5\mathbb{Z}$-equivariantly blow down $5$ skew lines all of which are defined over the field  such that we get either $\mathbb{P}^1 \times \mathbb{P}^1,$ or a quadric isomorphic to the Weil restriction of scalars of $\mathbb{P}^1.$ In the former case the rank of the Picard group of the blowdown is equal to $2,$ while in the latter case the rank of the Picard group of the blowdown is equal to $1.$  So we get $5$ points in general position. We show that such set of $5$ points on a quadric is unique up to automorphisms of the quadric.  So from the uniqueness of $5$ points in general position on quadrics up to automorphisms we get the uniqueness of smooth cubic surfaces up to isomorphism.

Now let us describe the structure of the paper.
In Section~\ref{section:projectiveline} we study subgroups of order $5$ in the automorphism group of $\mathbb{P}^1$ over finite fields of characteristic $2.$
In Section~\ref{section:quadricsurfaces} we study quadric surfaces and orbits of points on quadric surfaces under the action of the group of order $5.$  
In Section~\ref{section:Weyl} we study elements and subgroups of the Weyl group $W(\mathrm{E}_6).$
In Section~\ref{section:cubicsurface} we recall some basic facts about smooth cubic surfaces and discuss their automorphisms.
In Section~\ref{section:proofofmainresults} we prove Theorems~\ref{th:maximalaut2} and~\ref{th:maximalaut2z/5z}.
In Section~\ref{section:fermat} we discuss the Fermat cubic surface over finite fields of characteristic $2$ and prove Proposition~\ref{prop:Fermat} and Corollary~\ref{cor:Fermat}.

\vspace{0.5cm}

\textbf{Notation.} Let $X$ be a variety defined over a field $\mathbf{F}.$ If $\mathbf{F} \subset \mathbf{L}$ is an extension of $\mathbf{F},$ then we will denote by $X_{\mathbf{L}}$ the variety
$$
X_{\mathbf{L}}=X \times_{\mathrm{Spec}(\mathbf{F})}\mathrm{Spec}(\mathbf{L})
$$

\noindent over $\mathbf{L}.$ By $\overline{\mathbf{F}}$ we denote the algebraic closure of $\mathbf{F}.$ By $\mathbf{F}^{sep}$ we denote the separable closure of $\mathbf{F}.$ By $X(\mathbf{F})$ we denote $\mathbf{F}$-points on~$X.$ By $\mathfrak{S}_n$ we denote the symmetric group of degree $n,$ and by $\mathfrak{A}_n$ we denote the alternating group of degree $n.$

\vspace{5mm}

\textbf{Acknowledgment.} It is a pleasure for the author to express gratitude to C.~A.~Shramov for his care, for proposing this problem, a lot of discussion and constant attention. Also the author would like to thank A.~S.~Trepalin who taught the author his vision of cubic surfaces and their symmetries. Also the author thanks the referee for important remarks about the geometric explanation of some group-theoretical lemmas.

This work was performed at the Steklov International Mathematical Center and supported by the Ministry of Science and Higher Education of the Russian Federation (agreement no. 075-15-2022-265).  The author is a winner of the all-Russia mathematical August Moebius contest of graduate and undergraduate student papers and thanks the jury and the board for the high praise of his work.  The author was partially supported by Theoretical Physics and Mathematics Advancement Foundation “BASIS”.

\section{Projective line}\label{section:projectiveline}

In this section we prove some auxiliary lemmas concerning the automorphism group of the projective line over a finite field of characteristic $2.$ First of all, let us add the following obvious remark.

\begin{Remark}\label{remark:aboutfieldschar2}
The field $\mathbb{F}_{2^k}$ contains a non-trivial fifth root of unity if and only if~\mbox{$k=4l.$} Indeed, the former property is equivalent to the condition
$$
2^k-1 \equiv 0 \mod 5.
$$
\noindent This holds if and only if $k=4l.$ 
\end{Remark}

Let us give another obvious lemma about the fifth root of unity in finite fields of characteristic $2.$

\begin{Lemma}\label{lemma:xi+xi^4}
Let $\xi \in \overline{\mathbb{F}}_2$ be a non-trivial fifth root of unity. Then

\begin{enumerate}
\renewcommand\labelenumi{\rm (\roman{enumi})}
\renewcommand\theenumi{\rm (\roman{enumi})}

\item\label{lemma:xi+xi^4lie4k}
the elements $\xi^2+\xi^3$ and $\xi+\xi^4$ lie in $\mathbb{F}_{4^k};$ 

\item\label{lemma:xi+xi^4lnotie22k1}
the elements $\xi^2+\xi^3$ and $\xi+\xi^4$ do not lie in $\mathbb{F}_{2^{2k+1}};$

\item\label{lemma:xi+xi^4galois} 
if $\psi \in \mathrm{Gal}(\mathbb{F}_{4^{2k+1}}/\mathbb{F}_{2^{2k+1}}) \simeq \mathbb{Z}/2\mathbb{Z}$ is the generator of the Galois group, then 
$$
\psi(\xi+\xi^4)=\xi^2+\xi^3.
$$

\end{enumerate}

\end{Lemma}

\begin{proof}
First of all note that $4^k \equiv (-1)^k \mod 5.$ Using this we can prove assertions~\ref{lemma:xi+xi^4lie4k} and \ref{lemma:xi+xi^4lnotie22k1}. Indeed, assertion~\ref{lemma:xi+xi^4lie4k} follows from the equations
\begin{gather*}
(\xi+\xi^4)^{4^k}=\xi^{4^k}+\xi^{4^{k+1}}=\xi^{(-1)^k}+\xi^{(-1)^{k+1}}=\xi+\xi^4\\
(\xi^2+\xi^3)^{4^k}=\xi^{2\cdot 4^k}+\xi^{3\cdot 4^{k}}=\xi^{2\cdot(-1)^k}+\xi^{3\cdot(-1)^{k}}=\xi^2+\xi^3.
\end{gather*}

Assertions \ref{lemma:xi+xi^4lnotie22k1} and \ref{lemma:xi+xi^4lnotie22k1} follow from the fact that the Galois group of a finite extension of a finite field is generated by the Frobenius automorphism (see, for instance,~\cite[Chapter \rom{5}, Theorem 5.4]{Lang}).

\end{proof}

We shall now begin studying the automorphism group of projective line. We give some lemmas about subgroups of order $5$ in $\mathrm{PGL}_2(\mathbf{F})$ for a finite field $\mathbf{F}.$ First of all, let us start with the following obvious lemma.

\begin{Lemma}\label{cor:orderp1timesp1notdiv5}
The order of the group $\mathrm{PGL}_2(\mathbb{F}_{2^{2k+1}})$ is not divisible by~$5.$
\end{Lemma}

\begin{proof}
 The order of the group $\mathrm{PGL}_2(\mathbb{F}_{2^{2k+1}})$ is equal to
$$
2^{2k+1}(4^{2k+1}-1) \equiv (-1)^k \mod 5.
$$

\end{proof}

\begin{Lemma}\label{lemma:power5inpgl}
Let $\mathbf{F}$ be a finite field and consider the exact sequence
$$
1 \to \mathbf{F}^* \to \mathrm{GL}_2(\mathbf{F}) \xrightarrow{\rho} \mathrm{PGL}_2(\mathbf{F}) \to 1.
$$
\noindent  If $a \in  \mathrm{PGL}_2(\mathbf{F})$ satisfies $a^5=e,$ where $e$ is the neutral element in $ \mathrm{PGL}_2(\mathbf{F}),$ then there is a matrix $M \in \rho^{-1}(a)$ such that $M^5=E,$ where $E$ is the identity matrix.
\end{Lemma}

\begin{proof}
Let us consider some matrix $M \in \rho^{-1}(a).$ We have 
\begin{equation}\label{eq:m^5}
M^5=c \cdot E,
\end{equation}
\noindent where $c \in \mathbf{F}^*.$

 We divide the proof into two cases. First of all, consider the case when $\mathbf{F}$ does not contain non-trivial fifth root of unity. Then the homomorphism
$$
f \colon \mathbf{F}^* \to \mathbf{F}^*
$$
\noindent which sends the element $x \in \mathbf{F}^*$ to $f(x)=x^5$ is an isomorphism due to triviality of the kernel of $f$ and finiteness of $\mathbf{F}^*.$  So we can multiply~$M$ by some scalar $m \in \mathbf{F}^*$ such that 
$$
m^5M^5=E.
$$

 Now consider the case when  $\mathbf{F}$ contains a non-trivial fifth root of unity. In this case~$f$ is not an isomorphism. This means that in equality \eqref{eq:m^5} the constant $c$ is not necessary contained in the image of $f.$ If $c$ lies in the image of $f,$ like in the first case we can find $m \in \mathbf{F}^*$ such that $m^5M^5=E$ and we are done. So assume that~$c$ does not lie in the image of $f.$ The minimal polynomial $g(x) \in \mathbf{F}[x]$ of $M$ has to be of degree $2,$ because $\sqrt[5]{c}\xi \notin \mathbf{F}^*$ for any fifth root of unity $\xi \in \mathbf{F}.$  Besides,~$g(x)$ divides $x^5-c.$ Therefore, it has to be of the form
$$
g(x)=x^2-\sqrt[5]{c}(\xi^i+\xi^j)x+\sqrt[5]{c^2}\xi^{i+j},
$$
\noindent where $i,j \in \{0,1,2,3,4\}$ and $i \neq j.$ However, $\sqrt[5]{c^2}\xi^{i+j}$ does not lie in $\mathbf{F}.$ So in this case the matrix $M \in \mathrm{GL}_2(\mathbf{F})$ which satisfies the equality \eqref{eq:m^5} does not exist. This means that $c$ should lie in the image of $f.$ So as in the previous case we get that~\mbox{$\widetilde{M}^5=E$} for some $\widetilde{M} \in \rho^{-1}(a).$

\end{proof}

\begin{Cor}\label{cor:power5inpgl}
Let $\mathbf{F}$ be a finite field and consider the exact sequence
$$
1 \to \mathbf{F}^* \to \mathrm{GL}_2(\mathbf{F}) \xrightarrow{\rho} \mathrm{PGL}_2(\mathbf{F}) \to 1.
$$
\noindent  Then for any subgroup $G \simeq \mathbb{Z}/5\mathbb{Z}$ in  $\mathrm{PGL}_2(\mathbf{F})$ there is a subgroup $\widetilde{G} \simeq \mathbb{Z}/5\mathbb{Z}$ in~$\mathrm{GL}_2(\mathbf{F})$ such that $\rho(\widetilde{G})=G.$ In addition, if $\mathbf{F}$ contains a non-trivial fifth root of unity $\xi,$ then up to conjugation the group $\widetilde{G}$ can be generated by $\bigl(\begin{smallmatrix}1 & 0\\ 0 & \xi\end{smallmatrix}\bigr).$ In particular, the subgroup $G$ is unique up to conjugation.
\end{Cor}

\begin{proof}
This immediately follows from the fact that by Lemma \ref{lemma:power5inpgl} for the group
$$
G \subset \mathrm{PGL}_2(\mathbf{F}),
$$
\noindent which is isomorphic to $\mathbb{Z}/5\mathbb{Z},$ there is a matrix $M \in \mathrm{GL}_2(\mathbf{F})$ of order $5$ such that~$\rho(M)$ is a generator of $G.$  Since $\xi$ lies in $\mathbf{F}$ the matrix $M$ is diagonalizable. So up to conjugation and multiplication by $\xi^r$ for an integer $r$ we have $M^s=\bigl(\begin{smallmatrix}1 & 0\\ 0 & \xi\end{smallmatrix}\bigr)$ for some $s=1,2,3,4.$   This means that the group of order $5$ in $\mathrm{PGL}_2(\mathbf{F})$ is unique up to conjugation.

\end{proof}

\begin{Lemma}\label{lemma:repwithoutxi4^k}
Assume that $\mathbb{F}_{4^k}$ does not contain non-trivial fifth root of unity. Then a subgroup $G$ of order $5$ in the group $\mathrm{PGL}_2(\mathbb{F}_{4^k})$ is generated up to conjugation by the matrix
\begin{equation}\label{eq2:representationwithout5root}
\begin{pmatrix}
\xi+\xi^4 & 1\\
1 & 0
\end{pmatrix},
\end{equation}
\noindent where $\xi \in \overline{\mathbb{F}}_2$ is a non-trivial fifth root of unity. In particular, the subgroup $G$ is unique up to conjugation.

\end{Lemma}

\begin{proof}
Let $\sigma$ be a generator of the group $G \subset \mathrm{PGL}_2(\mathbb{F}_{4^k})$ of order $5.$ By Corollary~\ref{cor:power5inpgl} there is a matrix $M \in \mathrm{GL}_2(\mathbb{F}_{4^k})$ of order $5$ such that its image in $\mathrm{PGL}_2(\mathbb{F}_{4^k})$ by the natural projection is $\sigma.$ So the matrix $M$ should satisfy the equation
\begin{equation}\label{eq:x^5-1}
x^5-1=0.
\end{equation}
\noindent As the minimal polynomial of $M$ divides \eqref{eq:x^5-1} and $M$ is not a scalar matrix, the degree of this minimal polynomial is equal to $2.$   Let us fix a non-trivial fifth root of unity $\xi.$  We have
$$
x^5-1=(x-1)(x^2+(\xi+\xi^4)x+1)(x^2+(\xi^2+\xi^3)x+1).
$$
\noindent By Lemma~\ref{lemma:xi+xi^4}\ref{lemma:xi+xi^4lie4k} both $\xi+\xi^4$ and $\xi^2+\xi^3$ lie in $\mathbb{F}_{4^k}.$  Therefore, the matrix $M$ is conjugate to one of the matrices
$$
\begin{pmatrix}
\xi+\xi^4 & 1\\
1 & 0
\end{pmatrix} \; \text{or} \;
\begin{pmatrix}
\xi^2+\xi^3 & 1\\
1 & 0
\end{pmatrix},
$$

\noindent where the minimal polynomial of the first matrix is $x^2+(\xi+\xi^4)x+1$ and the minimal polynomial of the second matrix is $x^2+(\xi^2+\xi^3)x+1.$ Note that if $M$ is conjugate to $\bigl(\begin{smallmatrix}
\xi+\xi^4 & 1\\
1 & 0
\end{smallmatrix}\bigl),$ then $M^2$ is conjugate to $\bigl(\begin{smallmatrix}
\xi^2+\xi^3 & 1\\
1 & 0
\end{smallmatrix}\bigl),$ and vice versa.   Thus, up to conjugation the group of order $5$ in $\mathrm{PGL}_2(\mathbb{F}_{4^k})$ can be generated by the matrix~\eqref{eq2:representationwithout5root}, and so up to conjugation such group is unique.

\end{proof}

\section{Quadric surfaces}\label{section:quadricsurfaces}
In this section we study smooth quadric surfaces in $\mathbb{P}^3$ over finite fields of characteristic~$2.$  First of all, we refer the reader to~\cite[\S 8.1]{ShramovGor} or~\cite[Section 1.3]{Weil}  for the definition of the Weil restriction of scalars. We recall the result from the paper~\cite{ShramovVologodsky} about the classification of quadric surfaces in $\mathbb{P}^3.$

\begin{Lemma}[{cf.~\cite[Lemma 7.3]{ShramovVologodsky}}]\label{lemma:ShramovVologodsky}
Let $\mathbf{F}$ be an arbitrary field. Let $X$ be a del Pezzo surface of degree~\mbox{$K_X^2=8$} such that  $X_{\overline{\mathbf{F}}} \simeq \mathbb{P}^1_{\overline{\mathbf{F}}} \times \mathbb{P}^1_{\overline{\mathbf{F}}}.$ Then the surface $X$ is isomorphic either to  $ C \times C',$ where $C$ and $C'$ are two conics over $\mathbf{F},$ or to~\mbox{$ \mathrm{R}_{\mathbf{L}/\mathbf{F}}(D),$} where \mbox{$\mathbf{F} \subset \mathbf{L}$} is a separable quadratic extension of $\mathbf{F}$ and $D$ is a conic over $\mathbf{L}.$

The surface $X$ is isomorphic to a quadric surface in $\mathbb{P}^3$ if and only if $X \simeq C \times C$ or to \mbox{$ \mathrm{R}_{\mathbf{L}/\mathbf{F}}(C_{\mathbf{L}}),$} for some conic $C$ over $\mathbf{F}.$  In this case one has
$$
\mathrm{Aut}(X) \simeq \mathrm{Aut}^{\circ}(X) \rtimes \mathbb{Z}/2\mathbb{Z},
$$

\noindent where $ \mathrm{Aut}^{\circ}(X) \subset  \mathrm{Aut}(X)$ is a subgroup that consists of all automorphism of $X$ which act trivially on $\mathrm{Pic}(X_{\mathbf{F}^{\text{sep}}}).$ More precisely, we have 
$$
 \mathrm{Aut}^{\circ}(C \times C) \simeq  \mathrm{Aut}(C) \times  \mathrm{Aut}(C)
$$
\noindent and
$$
 \mathrm{Aut}^{\circ}( \mathrm{R}_{\mathbf{L}/\mathbf{F}}(C_{\mathbf{L}})) \simeq \mathrm{Aut}(C_{\mathbf{L}}).
$$

\end{Lemma}

However, as we are dealing with finite fields, by Chevalley--Warning theorem all conics are isomorphic to $\mathbb{P}^1$. So we modify Lemma \ref{lemma:ShramovVologodsky}   for the case when the field is finite.

\begin{Lemma}\label{lemma:propertyofrestrictionWeil}
Let $X$ be a del Pezzo surface of degree $8$ over a finite field $\mathbf{F}$ such that $X_{\overline{\mathbf{F}}} \simeq \mathbb{P}^1_{\overline{\mathbf{F}}} \times \mathbb{P}^1_{\overline{\mathbf{F}}}.$ Then $X$ is a quadric surface. 

 Moreover, if  the Picard number of $X$ is equal to $2,$ then $X \simeq \mathbb{P}^1 \times \mathbb{P}^1$ and 
$$
\mathrm{Aut}(X) \simeq \left( \mathrm{PGL}_2(\mathbf{F}) \times \mathrm{PGL}_2(\mathbf{F})\right) \rtimes \mathbb{Z}/2\mathbb{Z},
$$
where $\mathbb{Z}/2\mathbb{Z}$ interchanges the elements in the pairs $(a,b) \in   \mathrm{PGL}_2(\mathbf{F}) \times \mathrm{PGL}_2(\mathbf{F}).$

 Otherwise, if the Picard number of $X$ is equal to $1,$ then $X$ is isomorphic to the Weil restriction of scalars~\mbox{$\mathrm{R}_{\mathbf{L}/\mathbf{F}}(\mathbb{P}^1),$} where $\mathbf{F} \subset \mathbf{L}$ is a quadratic extension of $\mathbf{F}.$ The automorphism group of $X$ is 
\begin{equation}\label{eq2:autofQ}
\mathrm{Aut}(X)=\mathrm{PGL}_2(\mathbf{L}) \rtimes \mathbb{Z}/2\mathbb{Z},
\end{equation}

\noindent where $\mathbb{Z}/2\mathbb{Z}$ acts on $\mathrm{PGL}_2(\mathbf{L})$ as the Galois group $\mathrm{Gal}\left( \mathbf{L}/\mathbf{F} \right) \simeq \mathbb{Z}/2\mathbb{Z}.$

\end{Lemma}

\begin{Remark}\label{remark:descendautomorphism}
Let $Q$ be the Weil restriction of scalars~\mbox{$\mathrm{R}_{\mathbf{L}/\mathbf{F}}(\mathbb{P}^1),$} where~\mbox{$\mathbf{F} \subset \mathbf{L}$} is a separable quadratic extension of $\mathbf{F}.$ Let $\tau$ be a generator of the Galois group of this extension. Let 
$$
\phi=(\chi,\upsilon) \in \mathrm{PGL}_2(\mathbf{L}) \rtimes \mathbb{Z}/2\mathbb{Z}
$$
\noindent  be an automorphism of~$Q.$ Then this automorphism extends to the automorphism 
$$
\phi=(\chi, \tau(\chi),\upsilon) \in \left(\mathrm{PGL}_2(\mathbf{L}) \times \mathrm{PGL}_2(\mathbf{L})\right) \rtimes \mathbb{Z}/2\mathbb{Z}
$$
\noindent of $Q_{\mathbf{L}}=\mathbb{P}^1_{\mathbf{L}} \times \tau_*\mathbb{P}^1_{\mathbf{L}},$ where $\tau_*\mathbb{P}^1_{\mathbf{L}} \simeq \mathbb{P}^1$ is Galois conjugate to $\mathbb{P}^1_{\mathbf{L}}.$  This immediately follows from the equivalence of the category of the projective varieties over $\mathbf{F}$ and the category of the projective varieties over $\mathbf{L}$ with descent data.

\end{Remark}

We need the following definition.

\begin{definition}\label{def:generalposition}
We say that $5$ points $P_1,\ldots, P_5 \in \mathbb{P}^1 \times \mathbb{P}^1$ are \textit{in general position} if no $2$ points of this set  lie on a divisor of bidegree $(1,0)$ and~$(0,1)$ and no $4$ points of this set lie on a divisor of bidegree $(1,1).$ If $Q \subset \mathbb{P}^3_{\mathbf{F}}$ is a quadric surface over a field $\mathbf{F},$ then we say that $5$ points $P_1,\ldots, P_5 \in Q$ are \textit{in general position} if they are in general position on
$$
Q_{\overline{\mathbf{F}}} \simeq \mathbb{P}^1_{\overline{\mathbf{F}}} \times \mathbb{P}^1_{\overline{\mathbf{F}}}.
$$

\end{definition}

Now let us study  the sets of $5$ points in general position on a quadric surface over~$\mathbb{F}_{4^k}.$

\begin{Prop}\label{prop:4^kwithfifth}
 Let $P_1, \ldots, P_5$ be a set of  points in general position on~\mbox{$\mathbb{P}^1 \times \mathbb{P}^1$} over the field $\mathbb{F}_{4^{2k}}$ such that they form an orbit of some subgroup 
$$
\mathbb{Z}/5\mathbb{Z} \subset \mathrm{Aut}(\mathbb{P}^1 \times \mathbb{P}^1).
$$
\noindent Then the set of points $P_1, \ldots, P_5$ is unique up to automorphisms of $\mathbb{P}^1 \times \mathbb{P}^1.$
\end{Prop}

\begin{proof}

By Remark~\ref{remark:aboutfieldschar2} the field  $\mathbb{F}_{4^{2k}}$  contains a non-trivial fifth root of unity.  We need to define the embedding 
$$
\eta: \mathbb{Z}/5\mathbb{Z} \hookrightarrow \mathrm{Aut}(\mathbb{P}^1 \times \mathbb{P}^1) \simeq \left(\mathrm{PGL}_2(\mathbb{F}_{4^{2k}}) \times \mathrm{PGL}_2(\mathbb{F}_{4^{2k}}) \right) \rtimes \mathbb{Z}/2\mathbb{Z}.
$$
\noindent  By Corollary~\ref{cor:power5inpgl}  the projection of the image of the embedding on any copy of the group~$\mathrm{PGL}_2(\mathbb{F}_{4^{2k}})$ is conjugate to the element $\bigl(\begin{smallmatrix}1 & 0\\ 0 & \xi\end{smallmatrix}\bigr)$ for a non-trivial fifth root of unity~$\xi \in \mathbb{F}_{4^{2k}}.$

By the condition of general position of  $P_1, \ldots, P_5$ (i.e. no $2$ points lie on a divisor of bidegree $(1,0)$ and $(0,1)$) we get that $\mathbb{Z}/5\mathbb{Z}$  projects non-trivially in both copies of $\mathrm{PGL}_2(\mathbb{F}_{4^{2k}}).$  In other words, up to conjugation we have 
$$
\eta(\sigma)=
\left\langle\begin{pmatrix}
1 & 0\\
0 & \xi_1
\end{pmatrix}, 
\begin{pmatrix}
1 & 0\\
0 & \xi_2
\end{pmatrix}, \mathrm{id}
 \right\rangle,
$$
\noindent where $\sigma$ is a generator of $\mathbb{Z}/5\mathbb{Z}$ and $\xi_1$ and $\xi_2$ are non-trivial fifth root of unity.

 Without loss of generality we can assume that $P_1$ is $([1:1],[1:1]) \in \mathbb{P}^1 \times \mathbb{P}^1.$ Indeed, any point of the form 
\begin{equation}\label{eq:abcd}
 ([a:b],[c:d]) \in \mathbb{P}^1 \times \mathbb{P}^1,
\end{equation}
 \noindent where $a$, $b$, $c$ and $d$ are non-zero, can be obtained from the point  $([1:1],[1:1])$ by the action of the automorphism
$$
\left\langle\begin{pmatrix}
a & 0\\
0 & b
\end{pmatrix}, 
\begin{pmatrix}
c & 0\\
0 & d
\end{pmatrix}, \mathrm{id}
 \right\rangle.
$$
\noindent This automorphism commutes with $\eta(\sigma).$ If one of $a,$ $b,$ $c,$ $d$ is zero, then the orbit of the point \eqref{eq:abcd} is not in general position. By general position of the points~\mbox{$P_1,$ $P_2,$ $P_3,$ $P_4,$ $P_5$} (i.e. no $4$ points lie on a divisor of bidegree $(1,1)$) we get that $\xi_1 \neq \xi_2.$ Let us fix some non-trivial fifth root of unity $\xi.$ So up to conjugation we have the following embeddings:
\begin{gather*}
\eta_2(\sigma)=
\left\langle\begin{pmatrix}
1 & 0\\
0 & \xi
\end{pmatrix}, 
\begin{pmatrix}
1 & 0\\
0 & \xi^2
\end{pmatrix}, \mathrm{id}
\right\rangle,\\
\eta_3(\sigma)=
\left\langle\begin{pmatrix}
1 & 0\\
0 & \xi
\end{pmatrix}, 
\begin{pmatrix}
1 & 0\\
0 & \xi^3
\end{pmatrix}, \mathrm{id}
 \right\rangle,\\
\eta_4(\sigma)=
\left\langle\begin{pmatrix}
1 & 0\\
0 & \xi
\end{pmatrix}, 
\begin{pmatrix}
1 & 0\\
0 & \xi^4
\end{pmatrix}, \mathrm{id}
 \right\rangle.
\end{gather*}

The orbit of the point $P_1$  under the action of $\eta_2(\mathbb{Z}/5\mathbb{Z})$ is 
\begin{equation}\label{eta2}
([1:1],[1:1]); \; ([1:\xi],[1:\xi^2]); \; ([1:\xi^2],[1:\xi^4]); \; ([1:\xi^3],[1:\xi]); \; ([1:\xi^4],[1:\xi^3]).
\end{equation}

\noindent The orbit of the point $P_1$  under the action of $\eta_3(\mathbb{Z}/5\mathbb{Z})$ is 
\begin{equation}\label{eta3}
([1:1],[1:1]); \; ([1:\xi],[1:\xi^3]); \; ([1:\xi^2],[1:\xi]); \; ([1:\xi^3],[1:\xi^4]); \; ([1:\xi^4],[1:\xi^2]).
\end{equation}

\noindent The orbit of the point $P_1$  under the  action of $\eta_4(\mathbb{Z}/5\mathbb{Z})$ is 
\begin{equation}\label{eta4}
([1:1],[1:1]); \; ([1:\xi],[1:\xi^4]); \; ([1:\xi^2],[1:\xi^3]); \; ([1:\xi^3],[1:\xi^2]); \; ([1:\xi^4],[1:\xi]).
\end{equation}

\noindent One can see that the orbit \eqref{eta3} can be obtained by the action of $\mathbb{Z}/2\mathbb{Z}$ on \eqref{eta2} by interchanging two copies of $\mathbb{P}^1.$ Also note that the orbit \eqref{eta4} can be obtained from~$5$ points
\begin{equation}\label{eq:prohibitedorbit}
([1:1],[1:1]); \; ([1:\xi],[1:\xi]); \; ([1:\xi^2],[1:\xi^2]); \; ([1:\xi^3],[1:\xi^3]); \; ([1:\xi^4],[1:\xi^4])
\end{equation}
\noindent by acting with the element
$$
\left\langle\begin{pmatrix}
1 & 0\\
0 & 1
\end{pmatrix}, 
\begin{pmatrix}
0 & 1\\
1 & 0
\end{pmatrix},  \mathrm{id}
 \right\rangle \in  \left(\mathrm{PGL}_2(\mathbb{F}_{4^{2k}}) \times \mathrm{PGL}_2(\mathbb{F}_{4^{2k}}) \right) \rtimes \mathbb{Z}/2\mathbb{Z}.
$$

\noindent The orbit \eqref{eq:prohibitedorbit} is prohibited, as all points in the orbit lie on a divisor of bidegree~$(1,1).$ Thus the orbit \eqref{eta4} is also prohibited. Therefore, we get the only allowed orbit in this case.

\end{proof}

\begin{Prop}\label{prop:4^kwithoutfifth}
 Let $P_1, \ldots, P_5$ be a set of different points in general position on~$\mathbb{P}^1 \times \mathbb{P}^1$ over the field $\mathbb{F}_{4^{2k+1}}$ such that they form an orbit of some subgroup 
$$
\mathbb{Z}/5\mathbb{Z} \subset \mathrm{Aut}(\mathbb{P}^1 \times \mathbb{P}^1).
$$
\noindent  Then the set of points $P_1, \ldots, P_5$ is unique up to automorphisms of $\mathbb{P}^1 \times \mathbb{P}^1.$
\end{Prop}

\begin{proof}
By Remark~\ref{remark:aboutfieldschar2} the field $\mathbb{F}_{4^{2k+1}}$ does not contain non-trivial fifth root of unity.    Let $\sigma$ be a generator of $\mathbb{Z}/5\mathbb{Z}.$ By Lemma~\ref{lemma:repwithoutxi4^k}  the projection of the image of~$\sigma$ under the embedding in $\mathrm{Aut}(\mathbb{P}^1 \times \mathbb{P}^1)$  on any copy of~$\mathrm{PGL}_2(\mathbb{F}_{4^{2k+1}})$ is conjugate to the matrix \eqref{eq2:representationwithout5root}. By the condition of general position of  $P_1, \ldots, P_5$ (i.e. no $2$ points lie on a divisor of bidegree $(1,0)$ and~$(0,1)$) we get that the image of $\sigma$  projects non-trivially in both copies of $\mathrm{PGL}_2(\mathbb{F}_{4^{2k+1}}).$

As in the proof of Proposition \ref{prop:4^kwithfifth}   we can assume that 
$$
P_1=([1:1],[1:1]) \in \mathbb{P}^1 \times \mathbb{P}^1.
$$
\noindent By general position of $P_1,$ $P_2,$ $P_3,$ $P_4,$ $P_5$ (i.e. no $4$ points lie on a divisor of bidegree~$(1,1)$) the  projections in both copies of $\mathrm{PGL}_2(\mathbb{F}_{4^{2k+1}})$ are different. So up to conjugation the image of $\mathbb{Z}/5\mathbb{Z}$ is generated by the following elements  in the group~$\mathrm{Aut}(\mathbb{P}^1 \times \mathbb{P}^1):$
\begin{gather*}
\eta(\sigma)=
\left\langle\begin{pmatrix}
\xi+\xi^4 & 1\\
1 & 0
\end{pmatrix}, 
\begin{pmatrix}
\xi^2+\xi^3 & 1\\
1 & 0
\end{pmatrix}, \mathrm{id}
 \right\rangle, \\
\widetilde{\eta}(\sigma)=
\left\langle\begin{pmatrix}
\xi^2+\xi^3 & 1\\
1 & 0
\end{pmatrix}, 
\begin{pmatrix}
\xi+\xi^4 & 1\\
1 & 0
\end{pmatrix}, \mathrm{id}
 \right\rangle,
\end{gather*}

\noindent where $\xi \in \mathbb{F}_{16^{2k+1}}$ is a non-trivial fifth root of unity. 

 The orbit of $P_1$ under the action of $\eta(\sigma)$ can be obtained from the orbit of $P_1$ under the action of $\widetilde{\eta}(\sigma)$ by the action of $\mathbb{Z}/2\mathbb{Z}$ which interchanges two copies of~$\mathbb{P}^1.$  So we get that there is only one allowed orbit  up to automorphisms of $\mathbb{P}^1 \times \mathbb{P}^1$ in this case.

\end{proof}

Now we study the Weil restriction of scalars of $\mathbb{P}^1.$

\begin{Prop}\label{prop:24^k}
Let  $Q$ be a quadric surface over the field $\mathbb{F}_{2^{2k+1}}$ which is isomorphic to the Weil restriction of scalars $\mathrm{R}_{\mathbb{F}_{4^{2k+1}}/\mathbb{F}_{2^{2k+1}}}(\mathbb{P}^1).$  Let $P_1, \ldots, P_5$ be a set of different points in general position on~$Q$ over the field $\mathbb{F}_{2^{2k+1}}$  such that they form an orbit of some subgroup~\mbox{$\mathbb{Z}/5\mathbb{Z} \subset \mathrm{Aut}(Q).$}  Then the set of points $P_1, \ldots, P_5$ is unique up to automorphisms of $Q.$
\end{Prop}

\begin{proof}
 Let us study the possible sets of $5$ points in general position on $Q.$ By Lemma~\ref{lemma:repwithoutxi4^k} a generator  of $\mathbb{Z}/5\mathbb{Z}$  in  $\mathrm{PGL}_2(\mathbb{F}_{4^{2k+1}})$  is conjugate to the matrix~\eqref{eq2:representationwithout5root} for some non-trivial fifth root of unity $\xi \in \overline{\mathbb{F}}_2$.  Therefore, we get that up to conjugation the subgroup $\mathbb{Z}/5\mathbb{Z}$ in~$\mathrm{Aut}(Q)$ is generated by
$$
\alpha(\sigma)=
\left\langle\begin{pmatrix}
\xi+\xi^4 & 1\\
1 & 0
\end{pmatrix}, 
 \mathrm{id}
 \right\rangle.
$$

\noindent By Remark~\ref{remark:descendautomorphism} the automorphism $\alpha(\sigma)$ of $Q$ extends to the automorphism 
$$
\alpha(\sigma)_{\mathbb{F}_{4^{2k+1}}}=
\left\langle\begin{pmatrix}
\xi+\xi^4 & 1\\
1 & 0
\end{pmatrix}, 
\begin{pmatrix}
\xi^2+\xi^3 & 1\\
1 & 0
\end{pmatrix}, 
 \mathrm{id}
 \right\rangle
$$
\noindent on $Q_{\mathbb{F}_{4^{2k+1}}}$ since by Lemma~\ref{lemma:xi+xi^4}\ref{lemma:xi+xi^4galois} the elements $\xi+\xi^4$ and $\xi^2+\xi^3$ in $\mathbb{F}_{4^{2k+1}}$ are Galois conjugate.

 Without loss of generality we can assume that $P_1$ is $([1:1],[1:1]) \in Q_{\mathbb{F}_{2^{2k+1}}}.$ Indeed, any point of the form 
\begin{equation}\label{eq:abcdQ}
 ([a:b],[\psi(a):\psi(b)]) \in \mathbb{P}^1 \times \mathbb{P}^1,
\end{equation}
 \noindent where $a$ and $b$ are non-zero and $\psi \in \mathrm{Gal}(\mathbb{F}_{4^{2k+1}}/\mathbb{F}_{2^{2k+1}})$ is a non-trivial element, can be obtained from the point  $([1:1],[1:1])$ by the action of the automorphism
$$
\left\langle\begin{pmatrix}
a & 0\\
0 & b
\end{pmatrix}, 
 \mathrm{id}
 \right\rangle.
$$
\noindent This automorphism commutes with $\alpha(\sigma).$ If $a$ or $b$  is zero, then the orbit of the point \eqref{eq:abcdQ} is not in general position.

 Let us consider the orbit of the point 
$$
P_1=([1:1],[1:1]) \in Q_{\mathbb{F}_{2^{2k+1}}}
$$
\noindent under the action of the group generated by $\alpha(\sigma)_{\mathbb{F}_{4^{2k+1}}}.$ It is
\begin{gather*}
([1:1],[1:1]); \; ([\xi^2+\xi^3:1], [\xi+\xi^4:1]); \; ([0:1], [0:1]); \\
 ([1:0],[1:0]); \; ([\xi+\xi^4:1],[\xi^2+\xi^3:1]). 
\end{gather*}

\noindent It is not hard to see that the points of this quintuple are in general position. Therefore, we get that there is only one orbit of $5$ points  in general position  on~$Q$  up to the action of the group~$\mathrm{Aut}(Q) \simeq \mathrm{PGL}_2(\mathbb{F}_{4^{2k+1}}) \rtimes \mathbb{Z}/2\mathbb{Z}.$

\end{proof}

\section{The Weyl group of type $\mathrm{E}_6$}\label{section:Weyl}
In this section we prove some lemmas concerning the Weyl group $W(\mathrm{E}_6).$ At the beginning we recall  the structure of the group $W(\mathrm{E}_6).$

\begin{Lemma}[{\cite[\S3.12.4]{finitesimplegroups}}]\label{lemma:w=psuz/2z}
For the Weyl group $W(\mathrm{E}_6)$ we have
$$
W(\mathrm{E}_6) \simeq \mathrm{PSU}_4(\mathbb{F}_2) \rtimes \mathbb{Z}/2\mathbb{Z}
$$
\noindent with a non-trivial action of $\mathbb{Z}/2\mathbb{Z}$ on $\mathrm{PSU}_4(\mathbb{F}_2).$ 
\end{Lemma}

Now we recall the results from~\cite[Table 9]{Carter} and~\cite[Theorem 1]{SwDyer} about the conjugacy classes of the elements of order $2$, $5$ and $10$ in the Weyl group~$W(\mathrm{E}_6).$

\begin{Lemma}\label{lemma:conjugacyclassese6}
In Table~\ref{table:conjugacyclassese6} we list conjugacy classes of elements in the group~$W(\mathrm{E}_6)$ of orders $2$, $5$ and $10,$ their cardinalities, orders of their centralizers, collections of eigenvalues of the representation on the root system $\mathrm{E}_6.$

\begin{table}[H]
\begin{center}
\setlength\extrarowheight{5pt}
\begin{tabular}{|c|c|c|c|c|}
\hline
\text{Order of} & \text{Conjugacy class} & \text{Cardinality} & \text{Order of} & \text{Eigenvalues}\\
\text{element} &  &  &  \text{centralizer} & \\
 \hline
$2$ & $A_1$ & $36$  & $1440$ & $-1,1,1,1,1,1$  \\[5pt]
\hline
$2$ & $A_1^2$ & $270$ & $192$ & $-1,-1,1,1,1,1$ \\[5pt]
\hline
$2$ & $A_1^3$ & $540$ & $96$ & $-1,-1,-1,1,1,1$  \\[5pt]
\hline
$2$ & $A_1^4$ & $45$ & $1152$ & $-1,-1,-1,-1,1,1$  \\[5pt]
\hline
$5$ & $A_4$ & $5184$ & $10$ & $\xi,\xi^2,\xi^3,\xi^4,1,1$  \\[5pt]
\hline
$10$ & $A_4 \times A_1$ & $5184$ & $10$ & $\xi,\xi^2,\xi^3,\xi^4,-1,1$  \\[5pt]
\hline
\end{tabular}
\vspace*{3mm}
\caption{Table of conjugacy classes of elements of order $2$, $5$ and~$10$ in $W(\mathrm{E}_6)$.}
\label{table:conjugacyclassese6}
\end{center}
\end{table}
\noindent Here $\xi \in \mathbb{C}$ is a non-trivial fifth root of unity.

\end{Lemma}

\begin{Lemma}\label{lemma:centralizerconjugacyclassA4}
The centralizer of an element of order $5$ in $W(\mathrm{E}_6)$ is  isomorphic to~$\mathbb{Z}/10\mathbb{Z}.$
\end{Lemma}

\begin{proof}
By Lemma \ref{lemma:conjugacyclassese6} the order of the centralizer of an element $g$ of order $5$ in~$W(\mathrm{E}_6)$ is equal to $10.$ This means that the centralizer of $g$ in the Weyl group is isomorphic either to $\mathbb{Z}/10\mathbb{Z},$ or to a dihedral group of order $10,$ since these two groups are the only groups of order $10$ up to isomorphism. However, the second case is impossible because $g$ lies in its centralizer.

\end{proof}

\begin{Lemma}\label{lemma:a1anda4}
Let $\tau \in W(\mathrm{E}_6)$ be an element of order $10.$ Then $\tau^5$ lies in the conjugacy class $A_1.$

\end{Lemma}

\begin{proof}
This directly follows from Lemma~\ref{lemma:conjugacyclassese6}.   By Table~\ref{table:conjugacyclassese6} we see that the eigenvalues of an element lying in $A_4 \times A_1$ are
$$
\xi, \xi^2, \xi^3, \xi^4, -1,1,
$$
\noindent where $\xi \in \mathbb{C}$ is a non-trivial fifth root of unity. We have $g=\tau^5.$  So the eigenvalues of $\tau^5$ are
$$
-1,1,1,1,1,1
$$
\noindent which are the eigenvalues of elements in conjugacy class $A_1.$

\end{proof}

\begin{Cor}\label{lemma:a1commutewitha4}
Let $g \in W(\mathrm{E}_6)$ be an element of order $2$ which lies in the centralizer of an element of order $5.$ Then $g$ lies in the conjugacy class $A_1.$
\end{Cor}

\begin{proof}
Let $h$ be an element of order $5$ in $W(\mathrm{E}_6).$ Then as $g$ is in the centralizer of $h$ the element $gh$ is of order $10.$ By Lemma~\ref{lemma:conjugacyclassese6} there is only one conjugacy class~$A_4 \times A_1$ of elements of order $10$ in $W(\mathrm{E}_6).$  Therefore, by Lemma~\ref{lemma:a1anda4} the element $g$ of order $2$ lies in the conjugacy class $A_1.$

\end{proof}

Now we study subgroups in $\mathrm{PSU}_4(\mathbb{F}_2)$ and in the Weyl group $W(\mathrm{E}_6).$

\begin{Lemma}[{\cite[p.~26]{Atlas}}]\label{lemma:maxsubgroupspsu4}
The maximal proper subgroups in $\mathrm{PSU}_4(\mathbb{F}_2)$ are listed in Table~\ref{table:maximalsubgroups}.
\begin{table}[H]
\begin{center}
\setlength\extrarowheight{5pt}
\begin{tabular}{|c|c|}
\hline
\text{Subgroup} & \text{Order}\\
\hline
$\left(\mathbb{Z}/2\mathbb{Z}\right)^4 \rtimes \mathfrak{A}_5$ & $960$  \\[5pt]
\hline
$\mathfrak{S}_6$ & $720$ \\[5pt]
\hline
$\left(3^{1+2}_{+}\rtimes \mathbb{Z}/2\mathbb{Z}\right)_{\sbullet[.40]}\mathfrak{A}_4$ & $648$  \\[5pt]
\hline
$\left(\mathbb{Z}/3\mathbb{Z}\right)^3 \rtimes \mathfrak{S}_4$ & $648$  \\[5pt]
\hline
$\left( \mathbb{Z}/2\mathbb{Z}^{\, \sbullet[.40]}\left(\mathfrak{A}_4 \times \mathfrak{A}_4\right)\right)_{\sbullet[.40]}\mathbb{Z}/2\mathbb{Z}$ & $576$   \\[5pt]
\hline
\end{tabular}
\vspace*{3mm}
\caption{Table of maximal subgroups in $\mathrm{PSU}_4(\mathbb{F}_2)$.}
\label{table:maximalsubgroups}
\end{center}
\end{table}

\noindent Here by $ 3^{1+2}_{+}$ we denote the extraspecial group of order $27$ with exponent $3,$ by $G_{\sbullet[.40]} H$ we denote a group with normal subgroup isomorphic to $G$ and quotient is isomorphic to $H$ and by $G^{\, \sbullet[.40]} H$ we denote the group $G_{\sbullet[.40]} H$ which is not a split extension.   All  subgroups in Table~\ref{table:maximalsubgroups} are unique up to conjugation.
\end{Lemma}

\begin{Lemma}\label{lemma:a_6inpsu4f2}
The subgroup in $W(\mathrm{E}_6)$ isomorphic to $\mathfrak{A}_6$ is unique up to conjugation.
\end{Lemma}

\begin{proof}
First of all, let us prove that the subgroup of $W(\mathrm{E}_6)$ isomorphic to $\mathfrak{A}_6$ lies in $\mathrm{PSU}_4(\mathbb{F}_2).$ Let us consider the canonical projection
$$
p: W(\mathrm{E}_6) \simeq  \mathrm{PSU}_4(\mathbb{F}_2) \rtimes \mathbb{Z}/2\mathbb{Z} \to \mathbb{Z}/2\mathbb{Z}.
$$
\noindent Then the kernel of the homomorphism 
$$
p\mid_{\mathfrak{A}_6}: \mathfrak{A}_6 \to \mathbb{Z}/2\mathbb{Z}
$$
\noindent is the whole group $\mathfrak{A}_6$ since it is simple. This means that this $\mathfrak{A}_6$ lies in $\mathrm{PSU}_4(\mathbb{F}_2).$

Now let us prove that in $\mathrm{PSU}_4(\mathbb{F}_2)$ there is the only subgroup isomorphic to $\mathfrak{A}_6$ up to conjugation.  The subgroup in $\mathrm{PSU}_4(\mathbb{F}_2)$  isomorphic to $\mathfrak{A}_6$ should lie in some maximal subgroup of $\mathrm{PSU}_4(\mathbb{F}_2).$ By Lemma~\ref{lemma:maxsubgroupspsu4} the only maximal subgroup, the order of which is divisible by the order of $\mathfrak{A}_6,$ is isomorphic to $\mathfrak{S}_6.$ So $\mathfrak{A}_6$ has to lie in $\mathfrak{S}_6.$  Moreover, the subgroup in $\mathfrak{S}_6$ isomorphic to $\mathfrak{A}_6$ is unique. Therefore, as $\mathfrak{S}_6$ is unique up to conjugation in $ \mathrm{PSU}_4(\mathbb{F}_2),$ we get that $\mathfrak{A}_6$ is unique in $\mathrm{PSU}_4(\mathbb{F}_2)$ up to conjugation. As a consequence,   there is the only  $\mathfrak{A}_6$ in $W(\mathrm{E}_6)$ up to conjugation.

\end{proof}

\begin{Lemma}\label{lemma:galoispsu4}
The centralizers of the subgroups $\mathrm{PSU}_4(\mathbb{F}_2)$ and 
$$
\left(\mathbb{Z}/2\mathbb{Z}\right)^4 \rtimes \mathfrak{A}_5 \subset \mathrm{PSU}_4(\mathbb{F}_2)
$$
\noindent in the Weyl group $W(\mathrm{E}_6)$ are trivial.
\end{Lemma}

\begin{proof}

Let us consider the subgroup $G \simeq \left(\mathbb{Z}/2\mathbb{Z}\right)^4 \rtimes \mathfrak{A}_5$ in $\mathrm{PSU}_4(\mathbb{F}_2).$ It is enough to prove that the centralizer of $G$  in $W(\mathrm{E}_6)$ is trivial. There is an element of order~$5$ in~$G.$ Therefore, by Lemma~\ref{lemma:centralizerconjugacyclassA4} the centraliser of $G$ lies in $\mathbb{Z}/10\mathbb{Z}.$ However, the element of order~$5$ does not lie in the center of $G.$ This means that  the centralizer of $G$ lies in~$\mathbb{Z}/2\mathbb{Z}.$ Assume that the centralizer is isomorphic to $\mathbb{Z}/2\mathbb{Z}$ and let $h$ be its generator. By Lemma~\ref{lemma:centralizerconjugacyclassA4} the element $h$ lies in the conjugacy class $A_1.$ By Lemma~\ref{lemma:conjugacyclassese6}  the order of its centralizer is equal to $1440.$ But the order $|G|=960$ does not divide~$1440.$ Therefore, the centralizer of $G$ is trivial. Moreover, from this we get that the centralizer of $\mathrm{PSU}_4(\mathbb{F}_2)$ is also trivial.

\end{proof}

\section{Cubic surfaces}\label{section:cubicsurface}

In this section we introduce the notation related to smooth cubic surfaces and recall the description of their geometry. More details on smooth cubic surfaces can be found in~\cite[Chapter~9]{DolgClass} or~\cite[Chapter \rom{4}]{Manin}.

Let $S \subset \mathbb{P}^3$ be a smooth cubic surface over a field $\mathbf{F}.$  Then $S_{\overline{\mathbf{F}}}$ is a blowup of $6$ points 
$$
P_1, \; P_2, \; P_3, \; P_4, \; P_5, \; P_6
$$
\noindent in general position on $\mathbb{P}^2_{\overline{\mathbf{F}}}.$ In other words, there is a birational morphism
$$
\pi \colon S_{\overline{\mathbf{F}}} \to \mathbb{P}^2_{\overline{\mathbf{F}}}
$$
\noindent which blows down the preimages of points $P_1,\ldots, P_6.$  Denote by
$$
E_1, \; E_2, \; E_3, \; E_4, \; E_5, \; E_6
$$

\noindent the preimages of  $P_1,$ $P_2,$ $P_3,$ $P_4,$~$P_5$ and $P_6,$ respectively. The exceptional curves are lines on $S_{\overline{\mathbf{F}}} \subset \mathbb{P}^3_{\overline{\mathbf{F}}}.$ The self-intersection of any line on a smooth cubic surface is equal to~$-1.$ It is well known that there are exactly~$27$ lines on a smooth cubic surface over an algebraically closed field. We denote the other $21$ lines by $Q_i$ and~\mbox{$L_{ij}=L_{ji}$} for~$i,j \in \{1, \ldots, 6\}$ and $i \neq j.$ Denote by $H$ the pullback of the class of a line on~$ \mathbb{P}^2_{\overline{\mathbf{F}}}.$ Then the $27$ lines on the cubic surface $S_{\overline{\mathbf{F}}}$ are  the following effective divisors
\begin{gather*}
E_i \quad \text{for} \quad i\in \{1,\ldots,6\}; \\
L_{ij} \sim H-E_i-E_j \quad \text{for} \quad i \neq j; \; i,j\in \{1,\ldots,6\}; \\
Q_i \sim 2H+E_i-\sum_{j=1}^6 E_j \quad \text{for} \quad  i\in \{1,\ldots,6\}.
\end{gather*}

\noindent Their intersections are:
\begin{gather*}
E_i^2=Q_i^2=L_{ij}^2=-1 \quad \text{for} \quad  i \neq j;  \\
E_i \cdot E_j=Q_i \cdot Q_j=0 \quad \text{for} \quad  i \neq j;  \\
E_i \cdot Q_j=L_{ij} \cdot E_i=L_{ij} \cdot Q_i=1 \quad \text{for} \quad  i \neq j;  \\
L_{ij} \cdot L_{kl}=1,  \quad \text{where} \;  i,j,k,l \in \{1,\ldots,6\} \; \text{are pairwise distinct;}  \\
L_{ij} \cdot L_{jk}=0, \quad \text{where} \;  i,j,k \in \{1,\ldots,6\} \; \text{are pairwise distinct}. 
\end{gather*}

Let us consider the set 
$$
\Delta=\{\alpha \in \mathrm{Pic}(S_{\overline{\mathbf{F}}}) \mid K_{S_{\overline{\mathbf{F}}}}\cdot \alpha=0 \; \text{and} \; \alpha^2=-2  \}.
$$
\noindent Then this set is the root system $\mathrm{E}_6.$ The Weyl group $W(\mathrm{E}_6)$ acts on the Picard group~$\mathrm{Pic}(S_{\overline{\mathbf{F}}})$ by orthogonal transformations with respect to the intersection form.  This action restricts to the configuration of lines on $S_{\overline{\mathbf{F}}}.$ It fixes $ K_{S_{\overline{\mathbf{F}}}}$ and preserves the cone of effective divisors. 

Recall the following important theorem.

\begin{Th}[{see, for instance,~\cite[Corollary 8.2.40]{DolgClass}}]\label{th:autinw}
The automorphism group of a smooth cubic surface $S$ is isomorphic to a subgroup in $W(\mathrm{E}_6).$
\end{Th}

Although in the literature Theorem~\ref{th:autinw} is proved over fields of characteristic zero, exactly the same proof also works in the case of positive characteristic.

Using the classification of automorphism groups of smooth cubic surfaces over an algebraically closed field  we can say more about their orders. 

\begin{Lemma}\label{lemma:allautgroupinpsu}
Let $S$ be a smooth cubic surface. Then $\mathrm{Aut}(S) \leqslant 25\,920.$
\end{Lemma}

\begin{proof}
This directly follows from the classification of the automorphism groups which can be found in~\cite[Theorem~5.3]{Hosoh} in zero characteristic and  in~\cite[Theorem~1.1]{DolgachevDuncan} for positive characteristic.

\end{proof}

Let us go back to the conjugacy classes of elements in $W(\mathrm{E}_6)$ in the light of the action on the Picard group on a cubic surface.

\begin{Lemma}[{\cite[Table 7.1, column 8]{Banwait}}]\label{lemma:linesconjugacyclasses}
Let $S$ be a smooth cubic surface over an algebraically closed field. Consider an element $g \in  W(\mathrm{E}_6)$ of order~$2$ acting on the configuration of lines on $S.$ The following assertions hold.

\begin{enumerate}
\renewcommand\labelenumi{\rm (\roman{enumi})}
\renewcommand\theenumi{\rm (\roman{enumi})}

\item\label{it1} Assume that $g$ lies in the conjugacy class $A_1.$  Then there are $15$  lines  on $S$ which are invariant under the action of $g.$ Moreover, under the action of $g$ there are $6$ invariant pairs of skew lines on $S,$  and in any pair the lines are interchanged by $g.$ 

\item\label{it2} Assume that $g$ lies in the conjugacy class~$A_1^2.$   Then  there are $7$ lines on~$S$ which are invariant under the action of $g.$ Under the action of $g$ there are $8$ invariant pairs of skew lines on~$S,$   and in any pair the lines are interchanged by $g.$  Moreover, under the action of $g$ there are $2$ invariant pairs of intersecting lines on~$S,$  and  in any pair the lines are interchanged by $g.$ 

\item\label{it3} Assume that $g$ lies in the conjugacy class $A_1^3.$ Then  there are~$3$  lines which are invariant under the action of $g.$ Under the action of $g$ there are  $6$ invariant pairs of skew lines on $S,$ and  in any pair the lines are interchanged by $g.$   Moreover, under the action of $g$ there are $6$ invariant pairs of intersecting lines on $S,$ and  in any pair the lines are interchanged by $g.$ 

\item\label{it4} Assume that $g$ lies in the conjugacy class $A_1^4.$  Then  there are~$3$  lines which are invariant under the action of $g.$ Moreover, under the action of $g$ there are   $12$ invariant pairs of intersecting lines on $S,$   and  in any pair the lines are interchanged by $g.$

\end{enumerate}
\end{Lemma}

Now let us study the blowdown of a smooth cubic surface to the quadric in the light of automorphism groups.  From now on we use the following definition.

\begin{definition}\label{def:quadraticposit}
Let $S$ be a smooth cubic surface over a field $\mathbf{F}.$ We say that~$5$ skew lines $l_1, \ldots, l_5$ on $S$ are \textit{in quadratic position} if blowing them down over $\overline{\mathbf{F}}$ gives $\mathbb{P}^1_{\overline{\mathbf{F}}} \times \mathbb{P}^1_{\overline{\mathbf{F}}}.$
\end{definition}

We emphasize that in Definition~\ref{def:quadraticposit} it is required that each of the lines $l_1, \ldots, l_5$ is defined over $\mathbf{F}.$

\begin{Remark}\label{remark:blowdowngeneralposition}
By~\cite[Chapter~4, Proposition~2.7 and Exercise~2.9]{Huybrechts} the blowup of points in general position on a quadric surface~\mbox{$\mathbb{P}^1_{\overline{\mathbf{F}}} \times \mathbb{P}^1_{\overline{\mathbf{F}}}$} gives a smooth cubic surface. In other words, the preimages of points in general position on a quadric surface are the lines in quadratic position. Conversely, the image of the blowdown of lines in quadratic position on a smooth cubic surface is a set of points in general position on  $\mathbb{P}^1_{\overline{\mathbf{F}}} \times \mathbb{P}^1_{\overline{\mathbf{F}}}.$

\end{Remark}

Let $S$ be a smooth cubic surface over a field $\mathbf{F}.$ By~\cite[Theorem~23.8(\romsmall{2})]{Manin} the Galois group~$\mathrm{Gal}(\overline{\mathbf{F}}/\mathbf{F})$ acts on $\mathrm{Pic}(S_{\overline{\mathbf{F}}})$ and preserves the canonical class of $S$ and the intersection form. This means that there is a homomorphism 
$$
\mathrm{Gal}(\overline{\mathbf{F}}/\mathbf{F}) \to \mathrm{Aut}(\mathrm{Pic}(S_{\overline{\mathbf{F}}})) \subset W(\mathrm{E}_6).
$$
\noindent Let $\Gamma$ be the image of the Galois group~$\mathrm{Gal}(\overline{\mathbf{F}}/\mathbf{F})$ in  $W(\mathrm{E}_6).$ We remind the reader that  $\Gamma$ lies in the centralizer of  $\mathrm{Aut}(S),$  since any $\phi \in \mathrm{Aut}(S)$ is defined over $\mathbf{F}.$

Now we give auxiliary lemmas about the structures of conic bundle on del Pezzo surfaces.

\begin{Lemma}\label{lemma:conicbundleoncubic}
Let $S$ be a smooth cubic surface over an arbitrary field $\mathbf{F}.$ Assume that there is a line $L$ on $S_{\overline{\mathbf{F}}}$ which is defined over $\mathbf{F}.$ Then there is a structure of conic bundle over $\mathbb{P}^1$ 
$$
b \colon S \to \mathbb{P}^1
$$
\noindent on $S$ which is given by the pencil of planes on $\mathbb{P}^3_{\overline{\mathbf{F}}}$ passing through $L.$ This conic bundle has $5$ degenerate fibres over $\overline{\mathbf{F}}.$ Any degenerate fibre over $\overline{\mathbf{F}}$ is a union of two intersecting lines on $S_{\overline{\mathbf{F}}}.$

\end{Lemma}

\begin{proof}
This fact seems well-known, but for the sake of completeness let us give a proof of this fact. Let us consider the pencil of planes on $\mathbb{P}^3_{\overline{\mathbf{F}}}$ passing through $L.$ This pencil gives us a structure of conic bundle on $S$ over $\mathbf{F}.$ The base of this conic bundle is $\mathbb{P}^1,$ because $L$ is defined over $\mathbf{F},$ so there is a point $p \in L$ which is defined over $\mathbf{F}.$ So the point $b(p)$ is also defined over $\mathbf{F}.$

 By the very definition, any degenerate fibre over $\overline{\mathbf{F}}$ is a union of two intersecting lines on $S_{\overline{\mathbf{F}}}.$ So by computing the Euler characteristic of conic bundle we can see that there are exactly $5$ geometric degenerate fibres.

\end{proof}

\begin{Lemma}\label{lemma:uniqueconicbundlef1}
Let $\mathcal{F}$ be a blowup of $\mathbb{P}^2$ in a point. Then there is a unique structure of conic bundle on $\mathcal{F}.$
\end{Lemma}

\begin{proof}
This immediately follows from the structure of Mori cone of $\mathcal{F}.$

\end{proof}

Now we begin studying actions of order $5$ and $10$ on smooth cubic surfaces.

\begin{Lemma}\label{lemma:order10}
Let $S$ be a smooth cubic surface over a field $\mathbf{F}$ such that all $27$ exceptional lines are defined over $\mathbf{F}.$ Let $\delta \in W(\mathrm{E}_6)$ be an element of order $10$ acting on the configuration of lines on $S.$ Then
\begin{enumerate}
\renewcommand\labelenumi{\rm (\roman{enumi})}
\renewcommand\theenumi{\rm (\roman{enumi})}

\item\label{lemma:order10length1} there are no lines on $S$ which are invariant under the action of $\delta;$

\item\label{lemma:order10length2}  there is a pair of  lines on $S$ which is invariant under the action of $\delta,$ and, moreover, this pair consists of skew lines.

\end{enumerate}

\end{Lemma}

\begin{proof}
Let $D$ be a cyclic subgroup in $W(\mathrm{E}_6)$ generated by $\delta.$ Since the order of $D$ is $10,$ we get that the length of an orbit under the action of $D$ on the configuration of lines on $S$ can be $1,$ $2,$ $5$ or $10.$ Let us prove that there are no orbits of length~$1.$ Assume the converse. Let $l$ be a $D$-invariant line on $S.$ Then we can blow down $l$ and get the del Pezzo surface $S'$ of degree $4.$ Let 
$$
\pi \colon S \to S'
$$
\noindent be this blowdown. By~\cite[Corollary 8.2.40]{DolgClass} and~\cite[\S 3.1]{DolgachevDuncan} we have 
$$
\mathrm{Aut}(S') \subseteq W(\mathrm{D}_5) \simeq (\mathbb{Z}/2\mathbb{Z})^4 \rtimes \mathfrak{S}_5,
$$
\noindent where $W(\mathrm{D}_5)$ is the Weyl group of the root system $\mathrm{D}_5.$ It is not hard to see that in~$W(\mathrm{D}_5)$ there are no elements of order $10.$ Meanwhile, as  $l$ is a $D$-invariant line, we get that $D$ acts on the configuration of lines on $S',$ since $D$ preserves the intersection form on $\mathrm{Pic}(S_{\overline{\mathbf{F}}}).$ This means that $D \subset W(\mathrm{D}_5),$ which is a contradiction. So assertion~\ref{lemma:order10length1} is proven.

Now let us prove the existence of an orbit of length $2$ under the action of $D$ which is represented by $2$ skew lines. The existence of an orbit of length $2$ follows from~\ref{lemma:order10length1}. Assume that this orbit consists of intersecting lines. Then there is a unique line~$l$ which intersects both of these lines. But this means that $l$ is $D$-invariant which contradicts~\ref{lemma:order10length1}. So we get that the orbit of length $2$ consists of skew lines. So assertion~\ref{lemma:order10length2} is proven.

\end{proof}

\begin{Lemma}\label{lemma:quintuplewhichintersect2skewlines}
Let $S$ be a smooth cubic surface over a field $\mathbf{F}$ such that all $27$ exceptional lines are defined over $\mathbf{F}.$ Let $l_1$ and $l_2$ be two skew lines on $S.$ Then there are exactly $5$ lines on $S$ which intersect both $l_1$ and $l_2.$

\end{Lemma}

\begin{proof}
Let us consider the conic bundle 
$$
b \colon S \to \mathbb{P}^1,
$$
\noindent which is given by  the pencil of planes on $\mathbb{P}^3_{\overline{\mathbf{F}}}$ passing through $l_1.$  By Lemma~\ref{lemma:conicbundleoncubic} there are $5$ degenerate fibres on $b.$ Since $l_1$ does not intersect $l_2,$ we get that $l_2$ intersects any fibre of~$b$ in a point. This means that for any degenerate fibre the line~$l_2$ intersects one of the line in this fibre. Denote these lines in the degenerate fibres by~$\mathcal{L}_i$ for $i=1, \ldots, 5.$ Since $l_1$ intersects any of $\mathcal{L}_i$ for $i=1, \ldots, 5$ by construction of~$b,$ we get that $\mathcal{L}_1, \mathcal{L}_2, \mathcal{L}_3, \mathcal{L}_4 , \mathcal{L}_5$ is the required quintuple of lines. The uniqueness follows from the fact that all lines which intersect $l_1$ lie in the degenerate fibres.

\end{proof}

\begin{Lemma}\label{lemma:twoskewinvariantline5}
Let $S$ be a smooth cubic surface over a field $\mathbf{F}$ such that all $27$ exceptional lines are defined over $\mathbf{F}.$ Let $G$ be a group of order $5$ which acts on the configuration of lines on $S.$ Then there are two skew lines on $S$ which are $G$-invariant.
\end{Lemma}

\begin{proof}
The lengths of an orbit under the action of $G$ can be $1$ or $5.$ Since $27$ is not a multiple of $5,$ we get that there are at least two $G$-invariant lines. Denote them by $l_1$ and $l_2.$ Assume that they intersect each other. Then there is a unique line which intersects both $l_1$ and $l_2.$ This means that this line is also $G$-invariant. Therefore, we obtain three $G$-invariant lines. Since $27-3=24$ is not a multiple of~$5,$ we get that there are at least seven $G$-invariant lines. It is obvious, that there are two skew lines among these seven.

\end{proof}

\begin{Lemma}\label{lemma:5linesisorbit}
Let $S$ be a smooth cubic surface over a field $\mathbf{F}$ such that all $27$ exceptional lines are defined over $\mathbf{F}.$ Let $G$ be a subgroup of order $5$ in $\mathrm{Aut}(S).$ Then there is a quintuple of skew lines in quadratic position such that they form an orbit under the action of~$G.$ 
\end{Lemma}

\begin{proof}
 By Lemma~\ref{lemma:twoskewinvariantline5} there are two $G$-invariant skew lines on $S.$ Denote them by~$l_1$ and $l_2.$ By Lemma~\ref{lemma:quintuplewhichintersect2skewlines}  there are exactly $5$ lines on $S$ which intersect both $l_1$ and $l_2.$  Let us blow down this quintuple. Then we get a del Pezzo surface $S'$ of degree $8.$ Assume that it is a blowup of $\mathbb{P}^2$ in a point.  Then consider two conic bundles on $S,$ say $b_1$ and $b_2,$ such that 
$$
b_i \colon S \to \mathbb{P}^1
$$
\noindent is  given by the pencil of planes on $\mathbb{P}^3_{\overline{\mathbf{F}}}$ passing through $l_i$ for $i=1,2.$ The conic bundle $b_i$ gives the structure of $\mathbb{P}^1$-bundle on $S'$ for $i=1,2.$ So we get two different $\mathbb{P}^1$-bundle on $S'.$ By Lemma~\ref{lemma:uniqueconicbundlef1} this is impossible. Thus, we get that $S' \simeq \mathbb{P}^1 \times \mathbb{P}^1.$

\end{proof}

\begin{Cor}\label{cor:5linesisorbitequiv}
Let $S$ be a smooth cubic surface over a finite field $\mathbf{F}$ such that all~$27$ exceptional lines are defined over $\mathbf{F}.$ Let $G$ be a subgroup of order $5$ in $\mathrm{Aut}(S).$ Then there is a $\mathbb{Z}/5\mathbb{Z}$-equivariant blowdown $\pi:S \to \mathbb{P}^1 \times \mathbb{P}^1.$
\end{Cor}

\begin{proof}
By Lemma~\ref{lemma:5linesisorbit} there are $5$ skew lines in quadratic position which form a $G$-orbit. Therefore, blowing them down gives us a~\mbox{$\mathbb{Z}/5\mathbb{Z}$-equivariant} map 
$$
\pi:S \to X,
$$
\noindent where $X_{\overline{\mathbf{F}}} \simeq \mathbb{P}^1 \times \mathbb{P}^1.$ Since  all $27$ exceptional lines are defined on $S,$ we have the equality~\mbox{$\mathrm{rk}\mathrm{Pic}(S)=7,$} and, therefore, 
$$
\mathrm{rk}\mathrm{Pic}(X)=\mathrm{rk}\mathrm{Pic}(S)-5=2.
$$
\noindent By Lemma \ref{lemma:propertyofrestrictionWeil} we get that $X \simeq \mathbb{P}^1 \times \mathbb{P}^1.$

\end{proof}

\begin{Lemma}\label{lemma:z/5zz/2zquadric}
Let $S$ be a smooth cubic surface over $\mathbf{F}$ such that the image $\Gamma$ of the absolute Galois group $\mathrm{Gal}(\mathbf{F}^{sep}/\mathbf{F})$ in $W(\mathrm{E}_6)$ is isomorphic to~$\mathbb{Z}/2\mathbb{Z}.$ Let $G$ be a subgroup of order $5$ in $\mathrm{Aut}(S).$  Then there is a quintuple of skew lines in quadratic position  such that they form an orbit under the action of~$G.$ 

\end{Lemma}

\begin{proof}
Let $\gamma$  be a generator of $\Gamma$ and $g$ be a generator of $G.$ Since $\gamma$ and $g$ are commute in $W(\mathrm{E}_6),$ we can consider the group $D=G \times \Gamma \simeq \mathbb{Z}/10\mathbb{Z}.$ By Lemma~\ref{lemma:order10}\ref{lemma:order10length2} there is an orbit of length $2$ of skew lines $l_1$ and $l_2$ on $S_{\overline{\mathbf{F}}}$  under the action of $D$ on the configuration of lines on $S_{\overline{\mathbf{F}}}.$ Therefore, this orbit under the action of $D$ is also the orbit of length $2$ under the action of $\Gamma.$ By Lemma~\ref{lemma:quintuplewhichintersect2skewlines}  there are exactly $5$ lines on $S,$ say  $\mathcal{L}_1, \mathcal{L}_2, \mathcal{L}_3, \mathcal{L}_4, \mathcal{L}_5,$ which intersect both $l_1$ and $l_2.$ Since this quintuple of lines is unique we get that any $\mathcal{L}_i$ is~$\Gamma$-invariant, which means that all of them are defined over $\mathbf{F}.$ Moreover, this quintuple is $G$-orbit of length $5,$ since otherwise $G$ acts trivially on $S.$

 Let us  blow down this quintuple. We get a del Pezzo surface $S'$ of degree $8.$ Assume that $S'$ is a blowup of $\mathbb{P}^2$ in a point. Then there is a line on $S$ which is invariant under the action of $D.$ However, this contradicts Lemma~\ref{lemma:order10}\ref{lemma:order10length1}. So we get that $S'_{\overline{\mathbf{F}}} \simeq \mathbb{P}^1_{\overline{\mathbf{F}}} \times \mathbb{P}^1_{\overline{\mathbf{F}}}.$

\end{proof}

\begin{Cor}\label{cor:z/5zz/2zquadric}
Let $S$ be a smooth cubic surface over a finite field $\mathbf{F}$ such that the image $\Gamma$ of the absolute Galois group $\mathrm{Gal}(\mathbf{F}^{sep}/\mathbf{F})$ in $W(\mathrm{E}_6)$ is isomorphic to~$\mathbb{Z}/2\mathbb{Z}.$ Let~$G$ be a subgroup of order $5$ in $\mathrm{Aut}(S).$ Then there is a $\mathbb{Z}/5\mathbb{Z}$-equivariant blowdown $\pi:S \to Q,$ where $Q \simeq \mathrm{R}_{\mathbf{L}/\mathbf{F}}(\mathbb{P}^1)$ such that $\mathbf{F} \subset \mathbf{L}$ is a separable quadratic extension of $\mathbf{F}.$
\end{Cor}

\begin{proof}
By Lemma~\ref{lemma:z/5zz/2zquadric} there are $5$ skew lines in quadratic position  such that they form a $G$-orbit (in particular, all of them are defined over $\mathbf{F}$). Therefore, blowing them down gives us a~\mbox{$\mathbb{Z}/5\mathbb{Z}$-equivariant} map 
$$
\pi:S \to X,
$$
\noindent where $X_{\overline{\mathbf{F}}} \simeq \mathbb{P}^1 \times \mathbb{P}^1.$ As the group $\Gamma$ is not trivial, we obtain $\mathrm{rk}\mathrm{Pic}(S)\leqslant 6,$ and, therefore, 
$$
\mathrm{rk}\mathrm{Pic}(X)=\mathrm{rk}\mathrm{Pic}(S)-5=1.
$$
\noindent By Lemma \ref{lemma:propertyofrestrictionWeil} we get that $X \simeq Q,$ where $Q \simeq \mathrm{R}_{\mathbf{L}/\mathbf{F}}(\mathbb{P}^1)$ such that $\mathbf{F} \subset \mathbf{L}$ is a separable quadratic extension of $\mathbf{F}.$

\end{proof}

Now we give a simple lemma about the triviality of the image of the absolute Galois group in $W(\mathrm{E}_6)$ for smooth cubic surface with large automorphism group.

\begin{Lemma}\label{cor:galoispsu4}
Let $S$ be a smooth cubic surface over a field $\mathbf{F}$ such that its automorphism group is isomorphic either to $\mathrm{PSU}_4(\mathbb{F}_2),$ or to $\left(\mathbb{Z}/2\mathbb{Z}\right)^4 \rtimes \mathfrak{A}_5.$ Then the image $\Gamma$ of the absolute Galois group $\mathrm{Gal}(\mathbf{F}^{sep}/\mathbf{F})$ in $W(\mathrm{E}_6)$ is trivial.
\end{Lemma}

\begin{proof}
This directly follows from Lemma~\ref{lemma:galoispsu4}, since $\Gamma$ commutes with the automorphism group of $S.$

\end{proof}

Finally, we prove two lemmas about the image of the absolute Galois group in the Weyl group $W(\mathrm{E}_6)$ over $\mathbb{F}_{2^{2k+1}}.$

\begin{Lemma}\label{lemma:f24kgaloisgroupisnontrivial}
Let $S$ be a smooth cubic surface over $\mathbb{F}_{2^{2k+1}}$ such that $\mathrm{Aut}(S)$ contains a subgroup of order $5.$ Then the image of the absolute Galois group in $W(\mathrm{E}_6)$ is non-trivial.
\end{Lemma}

\begin{proof}
Assume the converse. Then all exceptional lines on $S_{\overline{\mathbb{F}}_2}$ are defined over the field~$\mathbb{F}_{2^{2k+1}}.$ By Lemma~\ref{lemma:5linesisorbit} we get that there is $\mathbb{Z}/5\mathbb{Z}$-equivariant blowdown to~\mbox{$\mathbb{P}^1 \times \mathbb{P}^1.$} However, by Lemma~\ref{cor:orderp1timesp1notdiv5} the order of the group $\mathrm{Aut}(\mathbb{P}^1 \times \mathbb{P}^1)$ is not divisible by $5.$

\end{proof}

\begin{Lemma}\label{cor:galoiss6}
Let $S$ be a smooth cubic surface over $\mathbb{F}_{2^{2k+1}}$ such that~\mbox{$\mathrm{Aut}(S) \simeq \mathfrak{S}_6.$} Then the image $\Gamma$ of the absolute Galois group $\mathrm{Gal}\left(\overline{\mathbb{F}}_2/\mathbb{F}_{2^{2k+1}}\right)$ in $W(\mathrm{E}_6)$ is isomorphic to $\mathbb{Z}/2\mathbb{Z}$ and generated by an element lying in the conjugacy class $A_1.$
\end{Lemma}

\begin{proof}
Let us consider an element $g \in \mathfrak{S}_6$ of order~$5$ in $W(\mathrm{E}_6).$ By Lemma~\ref{lemma:centralizerconjugacyclassA4} its centralizer in~$W(\mathrm{E}_6)$ is isomorphic to $\mathbb{Z}/10\mathbb{Z}.$  The element $g$ does not lie in the centralizer of $\mathfrak{S}_6 \subset W(\mathrm{E}_6),$ because it does not lie in the centre of $\mathfrak{S}_6.$ Therefore, the centralizer of $\mathfrak{S}_6$ in~$W(\mathrm{E}_6)$ lies in $\mathbb{Z}/2\mathbb{Z}.$ This means that $\Gamma \subseteq \mathbb{Z}/2\mathbb{Z}.$ By Lemma~\ref{lemma:f24kgaloisgroupisnontrivial} the group  $\Gamma$ is non-trivial. Therefore, the group $\Gamma$ is isomorphic to $\mathbb{Z}/2\mathbb{Z}.$

By Corollary~\ref{lemma:a1commutewitha4} the group $\Gamma$ is generated by an element lying in the conjugacy class $A_1.$

\end{proof}

\section{Proof of the main results}\label{section:proofofmainresults}
In this section we study automorphism groups of smooth cubic surfaces over all finite fields of characteristic $2$ and prove Theorems~\ref{th:maximalaut2} and~\ref{th:maximalaut2z/5z}. First of all, we prove the following lemmas about the automorphism group of maximal order of a smooth cubic surface over a finite field of characteristic $2.$

\begin{Lemma}\label{lemma:maximal4kPSU}
Let $S$ be a smooth cubic surface over $\mathbb{F}_{4^k}.$ Then $|\mathrm{Aut}(S)| \leqslant 25\, 920$ and in the case of equality one has $\mathrm{Aut}(S)\simeq \mathrm{PSU}_4(\mathbb{F}_2).$
\end{Lemma}

\begin{proof}

The inequality follows from Lemma~\ref{lemma:allautgroupinpsu}. Let us consider smooth cubic surfaces over $\mathbb{F}_{4^k}.$ By Theorem~\ref{th:autfermatf4} the automorphism group of the Fermat cubic surface over $\mathbb{F}_{4}$ is isomorphic to $\mathrm{PSU}_4(\mathbb{F}_2).$ This means that the automorphism group of maximal order of a smooth cubic surface over $\mathbb{F}_{4^k}$ is isomorphic to $\mathrm{PSU}_4(\mathbb{F}_2).$

\end{proof}

\begin{Lemma}\label{lemma:maximal24kS6}
Let $S$ be a smooth cubic surface over $\mathbb{F}_{2^{2k+1}}.$ Then $|\mathrm{Aut}(S)|\leqslant 720$ and in the case of equality  one has $\mathrm{Aut}(S)\simeq \mathfrak{S}_6.$
\end{Lemma}

\begin{proof}
Let $S$ be a smooth cubic surface over $\mathbb{F}_{2^{2k+1}}$ such that its automorphism group is of maximal order among smooth cubic surfaces over $\mathbb{F}_{2^{2k+1}}.$
Then by Theorem~\ref{th:cubicoverf2} we have $|\mathrm{Aut}(S)| \geqslant 720.$ By~\cite[Theorem 1.1]{DolgachevDuncan} this means that 
$$
\mathrm{Aut}(S_{\overline{\mathbb{F}}_2}) \simeq \mathrm{PSU}_4(\mathbb{F}_2).
$$
\noindent  Therefore, the group $\mathrm{Aut}(S)$ is a subgroup in $\mathrm{PSU}_4(\mathbb{F}_2).$ By Lemmas~\ref{lemma:maxsubgroupspsu4} and~\ref{lemma:allautgroupinpsu} we get $3$ candidates for being the largest automorphism group of a smooth cubic surface $S$:
$$
\mathfrak{S}_6, \; \mathrm{PSU}_4(\mathbb{F}_2) \; \text{and} \; \left(\mathbb{Z}/2\mathbb{Z}\right)^4 \rtimes \mathfrak{A}_5.
$$
\noindent The groups $\mathrm{PSU}_4(\mathbb{F}_2)$ and $\left(\mathbb{Z}/2\mathbb{Z}\right)^4 \rtimes \mathfrak{A}_5$  are not automorphism groups of a smooth cubic surface over $\mathbb{F}_{2^{2k+1}},$ because otherwise by Lemma~\ref{cor:galoispsu4} the image of the absolute Galois group in $W(\mathrm{E}_6)$ is trivial.  However, by Lemma~\ref{lemma:f24kgaloisgroupisnontrivial} it is impossible. This means that the largest automorphism group is isomorphic to $\mathfrak{S}_6.$

\end{proof}

Now we are ready to prove our main results. Let us start with Theorem~\ref{th:maximalaut2z/5z}.

\begin{proof}[Proof of Theorem~\ref{th:maximalaut2z/5z}] 

Let us prove assertion~\ref{th:maximalaut2F4^kz/5z}.  Since $\Gamma$ is trivial, all  lines on~$S_{\overline{\mathbf{F}}}$ are defined over $\mathbf{F}.$ By Lemma~\ref{lemma:5linesisorbit} there are $5$ skew lines in quadratic position such that they form an orbit under the action of $\mathbb{Z}/5\mathbb{Z} \subset \mathrm{Aut}(S).$  By Corollary~\ref{cor:5linesisorbitequiv} we can blow them down $\mathbb{Z}/5\mathbb{Z}$-equivariantly and get $\mathbb{P}^1 \times \mathbb{P}^1.$  By Remark~\ref{remark:blowdowngeneralposition} the blowdown gives $5$ points in general position on $\mathbb{P}^1\times \mathbb{P}^1.$ By Propositions~\ref{prop:4^kwithfifth} and~\ref{prop:4^kwithoutfifth} we get that the set of these~$5$ points is unique up to conjugation. Therefore, a smooth cubic surface such that its automorphism group contains $\mathbb{Z}/5\mathbb{Z}$ and $\Gamma$ is trivial is unique up to isomorphism. By Theorem~\ref{th:autfermatf4} we get that the automorphism group of the Fermat cubic surface contains $\mathbb{Z}/5\mathbb{Z}.$ Thus, $S$ is isomorphic to the Fermat cubic surface.

Now let us prove assertion~\ref{th:maximalaut2F24^kz/5z}. By Lemma~\ref{lemma:z/5zz/2zquadric} there are $5$ skew lines in quadratic position  such that they form an orbit under the action of $\mathbb{Z}/5\mathbb{Z} \subset \mathrm{Aut}(S).$ By Corollary~\ref{cor:z/5zz/2zquadric}  we can blow them down $\mathbb{Z}/5\mathbb{Z}$-equivariantly and get the Weil restriction of scalars $Q \simeq \mathrm{R}_{\mathbf{L}/\mathbf{F}}(\mathbb{P}^1),$ where $\mathbf{F} \subset \mathbf{L}$ is a quadratic extension of~$\mathbf{F}.$  By Remark~\ref{remark:blowdowngeneralposition} the blowdown gives us $5$ points in general position on $Q.$   By Proposition~\ref{prop:24^k} we get that the set of these~$5$ points is unique up to automorphism.  Therefore, a smooth cubic surface such that its automorphism group contains $\mathbb{Z}/5\mathbb{Z}$ and $\Gamma \simeq \mathbb{Z}/2\mathbb{Z}$ is unique up to isomorphism.  By Theorem~\ref{th:cubicoverf2} we get that the automorphism group of cubic surface \eqref{eq:cubicsurfacef2} contains $\mathbb{Z}/5\mathbb{Z}.$ Thus,~$S$ is isomorphic to cubic surface~\eqref{eq:cubicsurfacef2}.

\end{proof}

\begin{proof}[Proof of Theorem~\ref{th:maximalaut2}]

First of all, let us consider smooth cubic surfaces over $\mathbb{F}_{4^k}.$ By Lemma~\ref{lemma:maximal4kPSU} the maximal order automorphism group of such a surface is isomorphic to $\mathrm{PSU}_4(\mathbb{F}_2).$ It remains to show that a smooth cubic surface with such automorphism group is unique up to isomorphism. As $\mathrm{PSU}_4(\mathbb{F}_2)$ contains~$\mathbb{Z}/5\mathbb{Z}$ and  by Lemma~\ref{cor:galoispsu4} the image of the absolute Galois group in $W(\mathrm{E}_6)$ is trivial, we can apply Theorem~\ref{th:maximalaut2z/5z}\ref{th:maximalaut2F4^kz/5z} and get the desired.

Now let us consider smooth cubic surfaces over $\mathbb{F}_{2^{2k+1}}.$ By Lemma~\ref{lemma:maximal24kS6} the automorphism group of maximal order of such a surface is isomorphic to $\mathfrak{S}_6.$ It remains to show that a smooth cubic surface with such automorphism group is unique up to isomorphism. The group $\mathfrak{S}_6$ contains $\mathbb{Z}/5\mathbb{Z}.$ By Lemma~\ref{cor:galoiss6} the image of the absolute Galois group in $W(\mathrm{E}_6)$ is isomorphic to $\mathbb{Z}/2\mathbb{Z}.$ So we can apply Theorem~\ref{th:maximalaut2z/5z}\ref{th:maximalaut2F24^kz/5z} and get the desired.

\end{proof}

\section{Fermat cubic surface over fields of characteristic $2$}\label{section:fermat}

In this section we study automorphism group of the Fermat cubic surface over fields of characteristic $2$ and prove Proposition \ref{prop:Fermat} and Corollary~\ref{cor:Fermat}.

\begin{proof}[Proof of Proposition \ref{prop:Fermat}]
By Theorem~\ref{th:maximalaut2}\ref{th:maximalaut2F4^k} the automorphism group of the Fermat cubic surface over $\mathbb{F}_{4^k}$ is isomorphic to $\mathrm{PSU}_4(\mathbb{F}_2).$ Let us briefly recall the proof of this statement, which can be found in~\cite[\S 5.1]{DolgachevDuncan}, cf. also~\cite[\S 1.3 and 2.3]{Cheng}. Let us consider the Hermitian bilinear form on the $4$-dimensional vector space over $\mathbb{F}_4$
$$
H\left((x,y,z,t),(x',y',z',t')\right)=x\Phi(x')+y\Phi(y')+z\Phi(z')+t\Phi(t'),
$$
\noindent  where $\Phi: \mathbb{F}_4 \to \mathbb{F}_4$ is the Frobenius automorphism defined by $\Phi(\mathtt{x})=\mathtt{x}^2$ for $\mathtt{x} \in \mathbb{F}_4.$ Therefore, since the equation of the Fermat cubic surface is
$$
H\left((x,y,z,t),(x,y,z,t)\right)=x^3+y^3+z^3+t^3,
$$
\noindent we get that its automorphism group contains the group $\mathrm{PSU}_4(\mathbb{F}_2).$ So by Lemma~\ref{lemma:maximal4kPSU} this means that the automorphism group of the Fermat cubic surface over $\mathbb{F}_4$ and algebraic extensions of $\mathbb{F}_4$ is isomorphic to  $\mathrm{PSU}_4(\mathbb{F}_2).$

So let us consider the Fermat cubic surface over the field $\mathbb{F}_{2^{2k+1}}.$ Over its algebraic extension of degree $2$ the automorphism group of the Fermat cubic surface is isomorphic to $\mathrm{PSU}_4(\mathbb{F}_2).$ These are the elements $A$ in $\mathrm{PGL}_4(\mathbb{F}_4)$ such that 
\begin{equation}\label{eq:atopphia}
A^{\top}\Phi(A)=E,
\end{equation}
\noindent where $\Phi$ acts on the entries of $A,$ $E$ is the neutral element in $\mathrm{PGL}_4(\mathbb{F}_4)$ and $A^{\top}$ is the transpose of $A.$ 

Consider the algebraic extension $\mathbb{F}_{2^{2k+1}} \subset \mathbb{F}_{4^{2k+1}}.$ Let $\Gamma$ be the Galois group of this extension. Then the automorphism group of the smooth cubic surface over~$\mathbb{F}_{2^{2k+1}}$ is a subgroup 
$$
G=\mathrm{PSU}_4(\mathbb{F}_2)^{\Gamma} \subset \mathrm{PSU}_4(\mathbb{F}_2)
$$
\noindent of Galois invariant elements. Since the Galois group of a finite algebraic extension of finite fields is generated by the Frobenius automorphism we get that $\Gamma$ is generated by $\psi,$ i.e. 
$$
\psi(\mathtt{x})=\mathtt{x}^{2^{2k+1}} \;\; \text{for} \;\; \mathtt{x} \in  \mathbb{F}_{4^{2k+1}}.
$$
\noindent So we get that $G$ is a subgroup of the group $\mathrm{PGL}_4(\mathbb{F}_4)$ consisting of the elements~$A$ which satisfies the equation 
\begin{equation}\label{eq:a=phia}
A=\Phi(A).
\end{equation}

This means that the automorphism groups of the Fermat cubic surface over~$\mathbb{F}_2$ and over~$\mathbb{F}_{2^{2k+1}}$ are isomorphic. So it is enough to find the automorphism group of the Fermat cubic surface over $\mathbb{F}_2.$ By~\eqref{eq:atopphia} and~\eqref{eq:a=phia} these are the elements 
$$
M \in \mathrm{GL}_4(\mathbb{F}_2)
$$
\noindent which satisfy the equation $M^{\top}M=E.$ In other words, these are the matrices
$$
M=\begin{pmatrix}
m_{11} & m_{12} & m_{13} & m_{14}\\
m_{21} & m_{22} & m_{23} & m_{24}\\
m_{31} & m_{32} & m_{33} & m_{34}\\
m_{41} & m_{42} & m_{43} & m_{44}\\
\end{pmatrix}
 \in \mathrm{GL}_4(\mathbb{F}_2),
$$
\noindent  such that 
\begin{gather}
\sum_{i=1}^4m_{ij}=1 \;\; \text{for all} \;\; j=1,2,3,4 \quad \text{and} \label{eq:sumcolimn}\\
  \sum_{i=1}^4m_{ij}m_{ik}=0 \;\;  \text{for} \;\; j,k=1,2,3,4 \;\; \text{and} \;\;  j \neq k.\label{eq:orthogonal0}
\end{gather}

\noindent Using non-degeneracy of $M,$ by \eqref{eq:sumcolimn} we get that  any column of $M$ consists of either one $1,$ or three $1.$ By \eqref{eq:orthogonal0} we get that if one of the column of $M$ consists of one~$1,$ then the other columns also consist of one $1;$ and  if one of the column of~$M$ consists of three $1,$ then the other columns also consist of three $1.$  The group of such elements is isomorphic to~$\mathbb{Z}/2\mathbb{Z} \times \mathfrak{S}_4.$ This completes the proof.

\end{proof}

\begin{proof}[Proof of Corollary \ref{cor:Fermat}]
 By Theorem~\ref{th:maximalaut2}\ref{th:maximalaut2F24^k} the automorphism group of cubic~\eqref{eq:cubicsurfacef2} is isomorphic to $\mathfrak{S}_6,$ while by Proposition~\ref{prop:Fermat} the automorphism of the Fermat cubic surface is isomorphic to $\mathbb{Z}/2\mathbb{Z} \times \mathfrak{S}_4.$ This proves assertion~\ref{prop:Fermat24^k}.

The isomorphism between \eqref{eq:cubicsurfacef2} and the Fermat cubic surface over $\mathbb{F}_{4^k}$  is given by the collineation  
$$
[x:y:z:t] \mapsto [x+y:\omega x+\omega^2 y:z+t:\omega z+\omega^2 t],
$$
\noindent where $\omega \in \mathbb{F}_{4^k}$ is a non-trivial cube root of unity, which maps cubic surface \eqref{eq:cubicsurfacef2} to the Fermat cubic surface. This proves assertion~\ref{prop:Fermat4^k}.

\end{proof}

%
%
%

\providecommand{\bysame}{\leavevmode\hbox to3em{\hrulefill}\thinspace}
\providecommand{\MR}{\relax\ifhmode\unskip\space\fi MR }
\providecommand{\MRhref}[2]{%
  \href{http://www.ams.org/mathscinet-getitem?mr=#1}{#2}
}
\providecommand{\href}[2]{#2}

\end{document}